\documentclass[10pt]{amsart}
\title{Point modules of quantum projective spaces}
\author{Kevin De Laet}
\address{Department of Mathematics, University of Antwerp \\ 
 Middelheimlaan 1, B-2020 Antwerp (Belgium) \\ {\tt kevin.delaet@uantwerpen.be}}
\author{Lieven Le Bruyn}
\address{Department of Mathematics, University of Antwerp \\ 
 Middelheimlaan 1, B-2020 Antwerp (Belgium) \\ {\tt lieven.lebruyn@uantwerpen.be}}

\date{}
\usepackage{amsmath,amsfonts,array,amscd,amsthm,amssymb,stackrel,verbatim,wrapfig,subfig,float}
\usepackage{enumerate}
\usepackage{url}
\usepackage{tikz}
\usepackage[all]{xy}
\usetikzlibrary{arrows,matrix}
\usetikzlibrary{shapes.geometric}
\usetikzlibrary{decorations.markings} 

\tikzset{
  vertice/.style={circle,draw=black},
  decoration={markings,mark=at position 0.5 with {\arrow{>}}}
}

\newcommand{\wis}[1]{{\text{\em \usefont{OT1}{cmtt}{m}{n} #1}}}
\newcommand{\C}{\mathbb{C}}

\newcommand{\PP}{\mathbb{P}}

\theoremstyle{plain}
\newtheorem{theorem}{Theorem}

\newtheorem{proposition}[theorem]{Proposition}

\newtheorem{example}[theorem]{Example}

\numberwithin{equation}{section}
\begin{document}

\maketitle

\begin{abstract}
In this note we give an explicit description of the irreducible components of the reduced point varieties of quantum polynomial algebras.
\end{abstract}

Consider the quantum polynomial algebra on $k+1$-variables with quantum commutation relations
\[
A_Q = \C \langle u_o,u_1,\hdots,u_k \rangle/(u_i u_j - q_{ij} u_j u_i, 0 \leq i,j \leq k) \]
where the entries of the $k+1 \times k+1$ matrix $Q=(q_{ij})_{i,j}$ are all non-zero and satisfy $q_{ii}=1$ and $q_{ji} = q_{ij}^{-1}$.

As $A_Q$ is a graded connected, iterated Ore-extension, it is Auslander-regular of dimension $k+1$ and we can consider the corresponding non-commutative projective space $\mathbb{P}^k_Q = \wis{Proj}(A_Q)$ in the sense of \cite{ATV1}.

Recall from \cite{ATV1} that a {\em point module} $P=P_0 \oplus P_1 \oplus \hdots$ of $A_Q$ is a graded left $A_Q$-module which is cyclic (that is, generated by one element in degree $0$), critical (implying that normalizing elements of $A_Q$ act on it either as zero or a non-zero divisor) with Hilbert-series $(1-t)^{-1}$.

Hence, a point module $P$ is necessarily of the form $A_Q/(A_Q l_1 + \hdots + A_Q l_k)$ where the $l_i$ are linearly independent degree one elements in $A_Q$, and hence determines a unique point $x_P=\mathbb{V}(l_1,\hdots,l_k)$ in commutative $k$-dimensional projective space $\mathbb{P}^n = \wis{Proj}((A_Q)_1^{\ast})$.

In this note we will describe the reduced subvariety of $\mathbb{P}^k$, which is called {\em the point-variety of $A_Q$}
\[
\wis{pts}(A_Q) = \{ x_P \in \mathbb{P}^k~|~P~\text{a point module of $A_Q$} \}.
\]
We can approach this problem inductively.

\begin{proposition} For each of the generators $u_i$ of $A_Q$ we have
\[
\wis{pts}(A_Q) = (\wis{pts}(A_Q) \cap \mathbb{V}(u_i^{\ast})) \sqcup (\wis{pts}(A_Q) \cap \mathbb{X}(u_i^{\ast})) \]
and these pieces can be described as follows:
\begin{enumerate}
\item{$\wis{pts}(A_Q) \cap \mathbb{V}(u_i^{\ast}) = \wis{pts}(A_{\overline{Q}})$ where $\overline{Q}$ is the $k \times k$ matrix obtained from $Q$ after deleting the $i$-th row and column.}
\item{$\wis{pts}(A_Q) \cap \mathbb{X}(u_i^{\ast})$ is the affine variety
\[
\bigcap_{j,l \not= i} \mathbb{V}((r_{jl}-1)v_j^{\ast} v_l^{\ast}) \]
for the polynomial functions on $\mathbb{X}(u_i^{\ast})$: $v_j^{\ast} = u_j^{\ast}(u_i^{\ast})^{-1}$ and with $r_{jl} = q_{ij}q_{jl}q_{il}^{-1}$}
\item{In particular, $\wis{pts}(A_Q) = \mathbb{P}^k$ if and only if the rank of the matrix $Q$ is equal to one.}
\end{enumerate}
\end{proposition}

\begin{proof}
As $u_i$ is normalizing in $A_Q$ it acts either as zero on a point module $P$ or as a non-zero divisor. The point modules on which $u_i$ acts as zero are $\wis{pts}(A_Q) \cap \mathbb{V}(u_i^{\ast})$, correspond to point modules of the quantum polynomial algebra $A_Q/(u_i) \simeq A_{\overline{Q}}$ and are contained in the projective subspace $\mathbb{P}^{k-1} = \mathbb{V}(u_i^{\ast}) = \wis{Proj}((A_{\overline{Q}})_1^{\ast})$. This proves (1).

If $u_i$ acts as a non-zero divisor on the point module $P$, it extends to a graded module over the localization of $A_Q$ at the multiplicative system of homogeneous elements $\{ 1, u_i,u_i^2, \hdots \}$ and as this localization has an invertible degree one element it is a {\em strongly graded algebra}, see \cite[\S I.3]{NastaFVO}, and hence is a skew Laurent-polynomial extension
\[
A_Q[u_i^{-1}] \simeq (A_Q[u_i^{-1}])_0 [ u_i, u_i^{-1}, \sigma] \]
where $(A_Q[u_i^{-1}])_0$ is the degree zero part of the localization and $\sigma$ the automorphism on it induced by conjugation with $u_i$.

The algebra $(A_Q[u_i^{-1}])_0$ is generated by the $k$ elements $v_j = u_j u_i^{-1}$ and as we have the commuting relations $u_ju_i^{-1} = q_{ij} u_i^{-1} u_j$ we have
\[
v_j v_l = q_{ij} u_i^{-1} u_j u_l u_i^{-1} = q_{ij}q_{jl} u_i^{-1} u_l u_j u_i^{-1} = q_{ij}q_{jl}q_{il}^{-1} u_l u_i^{-1} u_j u_i^{-1} = q_{ij}q_{jl}q_{il}^{-1} v_l v_j \]
Therefore, $(A_Q[u_i^{-1}])_0$ is again a quantum polynomial algebra of the form $A_R$ where $R = (r_{jl})_{j,l}$ is the $k \times k$ matrix with entries
\[
r_{jl} = q_{ij}q_{jl}q_{il}^{-1} \]
Because $(A_Q[u_i^{-1}])_0$ is strongly graded, the localization $P[u_i^{-1}]$ (and hence the point module $P$) is fully determined by the one-dimensional representation $P[u_i^{-1}]_0$ of $(A_Q[u_i^{-1}])_0$, see \cite[\S I.3]{NastaFVO} or \cite[Proposition 7.5]{ATV1}.

One-dimensional representations of $A_R$ correspond to points $(a_j)_j \in \mathbb{A}^k$ (via the association $v_j \mapsto a_j$ for all $j \not= i$) satisfying all the defining relations of $A_R$, that is, they must satisfy the relations
\[
(1-r_{jl})a_j a_l = 0 \]
which proves (2).

As for (3), observe that $\wis{pts}(A_Q) \cap \mathbb{X}(u_i^{\ast}) = \mathbb{A}^k$ if and only if all the $r_{jl}=1$. This in turn is equivalent, by the definition of the $r_{jl}$ to
\[
\forall j,l \not= i~:~q_{jl} = q_{il} q_{ij}^{-1} \]
but then, any $2 \times 2$ minor of $Q$ has determinant zero as
\[
\begin{bmatrix} q_{ju} & q_{jv} \\ q_{lu} & q_{lv} \end{bmatrix} = \begin{bmatrix} \frac{q_{iu}}{q_{ij}} & \frac{q_{iv}}{q_{ij}} \\
\frac{q_{iu}}{q_{il}} & \frac{q_{iv}}{q_{il}} \end{bmatrix} \]
and the same applies to minors involving the $i$-th row or column of $Q$. Hence, $Q$ is of rank one (and so is $\overline{Q}$) finishing the proof.
\end{proof}

This result also allows us to describe the irreducible components of the point-varieties of quantum polynomial algebras directly. Take a quantum polynomial algebra on $n+1$ variables
\[
A_M = \C \langle x_0,x_1,\hdots,x_n \rangle / (x_ix_j - m_{ij} x_j x_i,~0 \leq i,j \leq n) \]
Consider the points $\delta_i = [\delta_{i0}:\hdots:\delta_{in}]$ in $\mathbb{P}^n = \wis{Proj}((A_M)_0^{\ast})$. For any $k+1$-tuple $(i_0,i_1,\hdots,i_k)$ with $0 \leq i_0 < i_1 < \hdots < i_k \leq n$ consider the $k$-dimensional projective subspace $\mathbb{P}(i_0,\hdots,i_k) \subset \mathbb{P}^n$ spanned by the points $\delta_{i_j}$. Also denote the $k+1 \times k+1$ minor of $M$ by
\[
M(i_0,\hdots,i_k) = \begin{bmatrix} 1 & m_{i_0 i_1} & \hdots & m_{i_0 i_k} \\
m_{i_1 i_0} & 1 & \hdots & m_{i_1 i_k} \\
\vdots & \vdots & \ddots & \vdots \\
m_{i_k i_o} & m_{i_k i_1} & \hdots & 1 \end{bmatrix} \]
With these notations, we deduce from Proposition~1:

\begin{theorem} The reduced point-variety of the quantum polynomial algebra $A_M$ is equal to
\[
\wis{pts}(A_M) = \bigcup_{rk(M(i_0,i_1,\hdots,i_k))=1} \mathbb{P}(i_0,i_1,\hdots,i_k) \subset \mathbb{P}^n \]
As a consequence, the union of all lines $\cup_{i,j} \mathbb{P}(i,j)$ is always contained in $\wis{pts}(A_M)$ and will be equal to it for generic $M$.
\end{theorem}

\begin{proof} As $\mathbb{P}(i_0,\hdots,i_k) = \wis{Proj}((A_Q)_0^{\ast})$ with $A_Q = A_M/(x_{j_1},\hdots,x_{j_{n-k}})$ where $\{ 0,1,\hdots,n+1 \} = \{ i_0,i_1,\hdots,i_k \} \sqcup \{ j_1,\hdots,j_{n-k} \}$, the description of $\wis{pts}(A_M)$ follows from Proposition~1.

As for the generic statement, consider the matrix $M$ with $m_{ij}=-1$ if $i \not= j$. Clearly, as any $3 \times 3$ minor
\[
M(i,j,k) = \begin{bmatrix} 1 & -1 & -1 \\ -1 & 1 & -1 \\ -1 & -1 & 1 \end{bmatrix} \]
has rank 3, we have that $\wis{pts}(A_M) = \cup_{i,j} \mathbb{P}(i,j)$. Alternatively, one can use the results of \cite{LeBruynCentral}. In this case, the algebra $A_M$ is a Clifford algebra over $\C[x_0^2,\hdots,x_n^2]$ with associated quadratic form $D=\wis{diag}(x_0^2,\hdots,x_n^2)$. In \cite{LeBruynCentral} it was shown that for a Clifford algebra, the point-variety is a double cover $\xymatrix{\wis{pts}(A_M) \ar@{->>}[r] & \mathbb{V}(minor(3,D))} \subset \mathbb{P}^n_{[x_0^2:\hdots:x_n^2]}$, ramified over $\mathbb{V}(minor(2,D))$. For the given $M$ one easily checks that $\mathbb{V}(minor(3,D))$ is one-dimensional, hence so is $\wis{pts}(A_M)$.
\end{proof}

We leave the combinatorial problem of determining which subvarieties of $\mathbb{P}^n$ can actually arise as a suggestion for further research. Not all unions as above can occur.

\begin{example} In $\mathbb{P}^3$ only two of the $\mathbb{P}^2$ (out of four possible) can arise in a proper subvariety $\wis{pts}(A_M) \subsetneq \mathbb{P}^3$. For example, take
\[
M = \begin{bmatrix} 1 & a & b & x \\ a^{-1} & 1 & a^{-1} b & c \\ b^{-1} & ab^{-1} & 1 & bca^{-1} \\ x^{-1} & c^{-1} & ba^{-1}c^{-1} & 1 \end{bmatrix} \]
then, for generic $x$ we have
\[
\wis{pts}(A_M) = \mathbb{P}(0,1,2) \cup \mathbb{P}(1,2,3) \cup \mathbb{P}(0,3) \]
but once we want to include another $\mathbb{P}^2$, for example, $\mathbb{P}(0,1,3)$ we need the relation $x=ac$ in which case $M$ becomes of rank one, whence $\wis{pts}(A_M) = \mathbb{P}^3$.
\end{example}

\end{document}